\documentclass[11pt,letterpaper]{article}
\usepackage{fullpage}
\usepackage{graphics} 
\usepackage{graphicx} 
\usepackage{tikz}
\usepackage{enumitem}
\usepackage{amsfonts,amsmath,amsthm,amssymb}
\usepackage{booktabs,color}
\usepackage{empheq,lscape}
\usepackage{mapping}
\usepackage[ruled,lined]{algorithm2e}
\usepackage[bookmarks=false,colorlinks]{hyperref}
\hypersetup{
 linkcolor=blue, 
 citecolor=red,
 urlcolor=magenta
 }

\begin{document}
\title{A Lagrange decomposition based Branch and Bound algorithm for the Optimal Mapping of Cloud Virtual Machines }
\author{Guanglei Wang \and Walid Ben-Ameur \and Adam Ouorou }
\date{07 June 2018}
\maketitle
\begin{abstract}
One of the challenges of cloud computing is to optimally and efficiently assign
    virtual machines to physical machines. The aim of telecommunication
    operators is to minimize the mapping cost while respecting constraints
    regarding location, assignment and capacity. In this paper we first propose an
    exact formulation leading to a 0-1 bilinear constrained problem.  Then we
    introduce a variety of linear cuts by exploiting the problem structure and
    present a Lagrange decomposition based B\&B algorithm to obtain optimal solutions
    efficiently.~Numerically, we show that our valid inequalities close over
    $80\%$ of the optimality gap incurred by the well-known McCormick
    relaxation, and demonstrate the computational advantage of the proposed B\&B algorithm with extensive numerical experiments. 
 
\end{abstract}

\section{Introduction}\label{intro}
Virtualization technology enables the emergence of cloud computing as a
flexible and on-demand service.
In a virtualization-based network 
the placement of Virtual Machines (VM) exerts significant influence on the computation and communication performance of cloud services~\cite{Papa13}. 
A considerable amount of investigations have been devoted to the 
optimal assignment of VMs to servers accounting for 
certain objectives and constraints. 

Google in 2012 proposed a challenge organized by the French Operational Research
and Decision Aid Society (ROADEF) and the European Operational Research society
(EURO), where a set of VMs needs to be assigned to a set of servers
to minimize the assignment cost while balancing the usage of servers
under several resource constraints.~As reported in~\cite{Murat2016}, the
proposal takes into account capacity constraints of servers regarding CPU,
memory, storage.~However, it~\emph{does not} include
bandwidth constraints respecting certain throughput requirements among VMs. Exact
formulations of this problem are Mixed Integer Linear Programs (MILP). To deal with
large scale problems, different heuristics are proposed. On the other hand, the authors in~\cite{meng2010:improving} introduce a traffic-aware virtual machine
placement model taking into account bandwidth constraints, which leads to a Quadratic Assignment Problem (QAP) for the solution of which  a two-tier heuristic algorithm is proposed. 

This research topic is also discussed in the context of \emph{Virtual Network
Embedding} (VNE), where virtual networks are required to be mapped to a physical network~\cite{fischer2013virtual} while respecting different constraints and objectives. 
For instance, Houidi et al.~\cite{Walid11} propose a MILP
model to solve the VM assignment problem and later this work is extended in~\cite{houidi2015exact} to jointly take into account 
energy-saving, load balancing and survivability objectives. 
The authors in~\cite{Papa13} present a MILP model and consider a
two-phase heuristic: a node mapping phase and link mapping phase. In the node
mapping phase, random rounding techniques~\cite{Rag87} are used to correlate
flow variables and binary variables. In the link mapping phase, decisions on the
mapping of virtual links are made by solving a Multi-Commodity Network Flow
(MCNF) problem. Later the authors in~\cite{coniglio2016} propose a chance constrained MILP formulation 
to handle the uncertain demand of different virtual networks and they propose a couple of heuristics based on MILPs
for its solution.~For more details about the VNE technology and related investigations, 
we redirect interested readers to \cite{fischer2013virtual,mijumbi2016} for comprehensive surveys. 

A recent thesis~\cite{mechtri2014:VM} studies the virtual network infrastructure
provision in a distributed cloud environment, where a 0-1 bilinear constrained
model taking into account bandwidth constraints is proposed. Heuristic methods
exploiting graph partition and bipartite graph matching techniques are proposed
for the solution procedure.

More recently, Fukunaga et al.~\cite{fukunaga2017:VMs} consider the assignment of VMs under capacity 
constraints aiming at minimizing certain connection cost. A centralized model and a distributed model are proposed for modeling 
the connection cost. In the former case, a root node is introduced and the connection cost is defined as the length of network links connecting all host servers and the root node. 
A couple of approximation algorithms are presented for cases of uniform and nonuniform requests respectively.
 However all VMs are assumed to be the same and bandwidth constraints are not considered. 

In spite of these efforts, few focus on mathematical programming methods in the presence of 
bandwidth constraints. In contrast to heuristic approaches (e.g., Genetic Programming, Tabu Search) whose performance is usually evaluated by simulations, a mathematical programming approach offers performance guarantees with proved lower and upper bounds. 
Furthermore, it benefits from off-the-shelf solvers that are being continually improved. Thus mathematical programming based methods deserve in-depth investigations.

In our previous paper~\cite{Wang2016:mapping}, we formulate the bandwidth constrained mapping problem as a 0-1 bilinear constrained problem and our numerical results demonstrate that the problem is computationally challenging even for a small number of VMs. This article aims at improving the computational performance by orders of magnitude. It extends the previous work by introducing some effective valid inequalities and proposes a Lagrange decomposition based Branch and Bound (B\&B) algorithm to accelerate the solution procedure. Contributions are summarized as follows.
\begin{enumerate}
\item We propose a compact model with a number of novel valid inequalities for the mapping problem.
\item We propose a Lagrange decomposition procedure for generating strong valid inequalities thus improving the continuous relaxation lower bound. 
\item We develop a B\&B algorithm to solve the mapping problem to global optimality. Various valid inequalities are used to strengthen the relaxation at each node dynamically. 
\item We conduct extensive numerical experiments showing the effectiveness of the proposed algorithm by orders of computational improvement.
\end{enumerate}
The rest of this paper is organized as follows.~In Section~\ref{sec:background}, we state the background of the mapping problem.~In Section~\ref{sec:generalmodel}, the mathematical formulation of the mapping problem is presented and a couple of reformulations involving 
strong valid inequalities are proposed. Section~\ref{sec:BB} is dedicated to a B\&B algorithm where lower and upper bounding procedures are elaborated in detail. In Section~\ref{sec:numerical}, we evaluate the effectiveness of valid inequalities and the proposed B\&B algorithm.~Finally,~concluding 
remarks follow in Section~\ref{sec:conclusion}.

\section{Problem background}\label{sec:background}
VMs play an important role in a cloud computing environment. A customer's
request consists of a number of VMs, which are allocated on servers to execute
a specific program.{
Without special restrictions, a server can usually run multiple VMs simultaneously.

In order to improve the utility of data center resources, VMs should be
dynamically started or stopped and sometimes live migration should be conducted,
i.e.,~move a VM from one server to another. Thus virtual communications should also be
mapped to the physical network.

The focus of this paper is on the assignment of virtual resources to a given physical network. 
The solution to this problem is how to map the cloud resources and
which servers and links should be used subject to certain
hard constraints, e.g.~resource capacity constraints, traffic routing constraints.
 \begin{figure}[!htpb]
 \centering
 \includegraphics[scale=0.35]{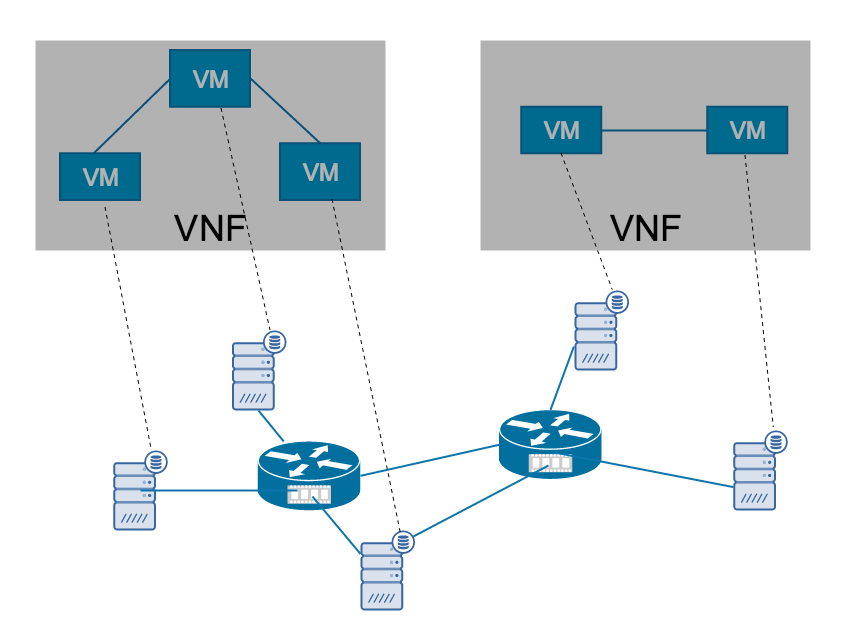}
 \caption{An illustration of the mapping procedure}
 \label{fig:mapping}
 \end{figure}

Figure~\ref{fig:mapping} illustrates the mapping problem involving two virtual
requests. One has three VMs and the other has two. In addition, VMs within the
same request communicate with each other. Dotted lines show 
a feasible mapping solution that VMs of each request are mapped to different severs and the communication
throughput between each pair of VMs is routed between the corresponding servers
that VMs are mapped to. 

Furthermore, one may need to be aware that the communication throughput between 
two servers (which host VMs) should route on a single path, as
multi-path routing may cause discrepancies among the arrivals of data at the destination.~So we assume that for each origin-destination (O-D) pair the corresponding traffic is routed on a shortest path. 
Since VMs within a request often communicate, we also assume that the request graph induced by VMs and virtual links is connected. 



\section{Formulations}\label{sec:generalmodel}
Recall that a virtual request consists of a set of virtual machines and their mutual virtual
communications. Therefore, we may represent each virtual request as
a directed graph. Our goal is to map such graphs to a physical network.
Henceforth, we will use the following notation to construct mathematical expressions. 
\par \begin{itemize}[label={}]
 \item [] Sets
\begin{itemize}[label={}]
 \item {\makebox[3cm][l]{$R$} set of virtual request} 
 \item {\makebox[3cm][l]{$H = (S,E)$} the connected graph of a physical network} 
  \item {\makebox[3cm][l]{$S$} set of servers in the physical network} 
 \item {\makebox[3cm][l]{$E$} set of undirected edges in the physical network}
 \item {\makebox[3cm][l]{$G^r = (V^r, L^r)$} a graph of virtual network for request $r \in R$} 
 \item {\makebox[3cm][l]{$V^r$} set of VMs of request $r$} 
 \item {\makebox[3cm][l]{$L^r$} set of undirected virtual links of request $r$} 
\end{itemize}
 \item [] Parameters
 \begin{itemize}[label={}]
 \item \makebox[3cm][l]{$c^{ri}$} required CPU of VM $i\in V^r$
 \item \makebox[3cm][l]{$m^{ri} $ } required memory of VM $i\in V^r$
 \item \makebox[3cm][l]{$C_k$} CPU cores of server $k$
 \item \makebox[3cm][l]{$M_k$ } memory capacity of server $k$
 \item \makebox[3cm][l]{$F_k$ } fixed cost of server $k\in S$
 \item \makebox[3cm][l]{$A_k$ } additional cost of server $k$ imposed from CPU loads
 \item \makebox[3cm][l]{$f^{rij}$ } required throughput associated with
 logical link $(i,j)\in L^r $
 \item \makebox[3cm][l]{$B_e$ } bandwidth of edge $e\in E$
 \item \makebox[3cm][l]{$W_e$ } fixed cost of edge $e\in E$ 
 \item \makebox[3cm][l]{$P_{kp}$ } shortest $k$-$p$ path, $(k,p)\in S\times
 S: k\neq p$
 \end{itemize}
 \item[] Variables
 \begin{itemize}[label={}]
 \item \makebox[3cm][l]{$x^{ri}_k\in \zeroone$} 1 if VM $i$ of request $r$ is mapped to server $k$
 \item \makebox[3cm][l]{$\theta_k \in \zeroone$ } 1 if server $k$ is used (switched on)
 \item \makebox[3cm][l]{$\phi_e \in \zeroone$ } 1 if edge $e$ is used (switched on)
 \end{itemize}
\end{itemize}
Notice that for a parameter or a variable, its subscripts (if it has) are
associated with physical resources while its superscripts (if it has) are
associated with virtual resources. For sake of convenience sometimes variables and parameters are 
presented in vector form where they are marked in bold. For instance $\vx^r$ represents
$\bigcurly{x^{ri}_k: i \in S, i \in V^r}$.

We construct the exact mathematical model of the mapping problem as follows
\begin{alignat}{3}
 \min~~& \sum\limits_{k\in S} 
 F_k\theta_k+\sum\limits_{k\in S}A_k\sum\limits_{r\in R}\sum\limits_{i\in
 V^r}c^{ri}x^{ri}_k+\sum\limits_{e\in E}W_e\phi_e&& \tag{$\bbP$} \label{bbP}\\
 \textrm{s.t.~~} &\sum_{k\in S} x^{ri}_k =1 &&\forall r\in R, i\in
 V^r\tag{AC}\label{Assign}\\
 &\sum\limits_{i\in V^r} x^{ri}_k \leq \theta_k && \forall
r\in R, \forall k\in S\tag{LC}\label{Location}\\
 &\sum\limits_{r\in R}\sum\limits_{i\in V^r} c^{ri}x^{ri}_k \leq
C_k\theta_{k}, && \forall k\in S\tag{KP}\label{KP1}\\
&\sum\limits_{r\in R}\sum\limits_{i\in V^r} m^{ri}x^{ri}_k \leq
M_k\theta_k && \forall k\in S\tag{KP'}\label{KP2}\\
&\sum\limits_{r\in R}\sum\limits_{\substack{k,p\in S:\\k \neq p,e \in
P_{kp}}}\sum\limits_{\substack{\{i,j\}\in L^r}} f^{rij}x^{ri}_kx^{rj}_p
\leq B_e\phi_e && \forall e\in E \tag{QC}\label{QC}\\
 &\theta_k,\phi_e, x^{ri}_k\in \{0,1\}  && \forall r\in R,i\in V^r, k\in S,
 e\in E. \tag{BC}\label{BC}
\end{alignat}
Henceforth we will use~\eqref{bbP} to represent this model and we interpret it as follows
\begin{itemize}
\item The objective is to minimize the total cost, which is
additively composed of three terms: the fixed cost incurred by switching
on servers, the additional cost coming from the CPU load, and the fixed cost from
the usage of links. We model the additional cost induced by CPU load as a linear
function to represent the fact that CPU is usually categorized as load
dependent resource while memory is load independent~\cite{Karve06}.
\item Constraints \eqref{Assign} mean that each virtual machine must be mapped to a single server.
 Constraints \eqref{Location} model the fact that virtual machines are
 usually mapped separately in a cloud environment due to some practical
 issues, e.g.,~security, reliability. 
\item Constraints (\ref{KP1}, \ref{KP2}) are~{knapsack constraints}. They ensure that for each server, the
 aggregated required CPU, memory resource cannot exceed its limits. 
 Constraints \eqref{QC} emphasis the fact that for each edge, the aggregated
 throughputs on the edge cannot exceed the bandwidth. 
\end{itemize} 
Before the solution procedure, we briefly analyze problem structure. First, the combination of integrality constraints and bilinear constraints (nonconvex)
makes the problem rather difficult.~Second, the bilinear products appearing in the formulation are dense.~Third, the problem aggregates features of the knapsack problem with multiple constraints
and the QAP problem. In fact the authors~\cite{Amaldi2016} show that the mapping problem of the form \eqref{bbP} is strongly NP-hard even
if $\abs{R}=1$ as there is a polynomial time reduction from the maximum stable set
problem. In what follows, we will focus on the solution procedure of ~\eqref{bbP}. 
\subsection{Reformulations}
Model $\eqref{bbP}$ is a 0-1 bilinear constrained problem. A fundamental
idea to deal with such problems is \textit{lifting} it to a higher dimensional
space~\cite{Burer2012,Burer2012MIQCP}.~Introducing new variables $y^{rij}_{kp}$
and enforcing $y^{rij}_{kp}=x^{ri}_kx^{rj}_p$ for each $(r,i,j,k,p):i\neq j,
k\neq p$, we lift the problem to a higher dimensional space and lead to a
MIP, at the expense of introducing non-convex
equations. 
Simple convex relaxations can be achieved by linearization techniques. 
In this paper, we adopt the well-known McCormick inequalities~\cite{McCormick1976}.
Specifically, for each $(r,i,j,k,p): i\neq j, k\neq p$, we approximate the equation $y^{rij}_{kp} = x^r_{ik}x^r_{jp}$ using the following four inequalities
 \begin{subequations}
 \begin{align}
 x^{ri}_k+x^{rj}_p - 1\leq y^{rij}_{kp}\label{linear1}\\
 y^{rij}_{kp}\leq x^{ri}_k\label{linear2}\\
 y^{rij}_{kp}\leq x^{rj}_p \label{linear3}\\
 y^{rij}_{kp}\geq 0.\label{linear4}
 \end{align}
\end{subequations}
And~\eqref{QC} is linearized as 
\begin{equation}\tag{QCL}\label{QCL}
\sum\limits_{r\in R}\sum\limits_{\substack{k,p\in S:\\k \neq p,e \in
P_{kp}}}\sum\limits_{\substack{\{i,j\}\in L^r}} f^{rij}y^{rij}_{kp}
\leq B_e\phi_e, \;  \forall e\in E 
\end{equation}
With this relaxation, we can convert the 0-1 bilinear model~\eqref{bbP} to a MIP
\begin{equation}\tag{$\bbP_{\mathrm{MC}}$}\label{model:MC}
 \mathrm{\bbP_{MC}}: \left\{   \eqref{Assign}, \eqref{Location}, \eqref{KP1}, \eqref{KP2}, \eqref{QCL}, \eqref{BC}\right\}\cap \left\{(\ref{linear1}-\ref{linear4})\right\}.
\end{equation}
The numerical results in~\cite{Wang2016:mapping} show that the bound provided by the continuous relaxation of~\eqref{model:MC} is weak. 
To strengthen the formulation, we need stronger valid inequalities. An important
subset of those can be derived by the Reformulation-Linearization-Technique (RLT)~\cite{Sherali1990}, which is a general framework for
generating valid inequalities in higher dimensional space for non-convex discrete and continuous formulations.
We employ the RLT for the assignment constraints \eqref{Assign} producing
$\sum\limits_{r\in R}\abs{V^r}(\abs{V^r}-1)\abs{S}$ linear equations:
\begin{equation} \tag{AC$_{\mathrm{RLT}}$}\label{Assignrlt} 
\sum\limits_{k\in S:k\neq p}y^{rji}_{pk} =x^{rj}_{p}~ \quad \forall (r, i, j, p): i\neq j.
 \end{equation}
\begin{prop}\label{prop:redudant}
 Constraints (\ref{linear1}-\ref{linear3}) are implied by
\eqref{Assignrlt}.
\end{prop}

\begin{proof}
 It is obvious to see that~\eqref{Assignrlt} imply \eqref{linear2},\eqref{linear3}. 
 Regarding $\eqref{linear1}$, for each $(r,i,j,k,p):k< p,i\neq j$, we have
 \begin{align*}
 x^{ri}_{k}+x^{rj}_{p}-1
 &= y^{rij}_{kp}+\sum\limits_{p'> k:p'\neq p} y^{rij}_{kp'} +\sum\limits_{p'< k} y^{rji}_{p'k}+x^{rj}_{p}-1\\
 &= y^{rij}_{kp}+\sum\limits_{p'> k:p'\neq p} y^{rij}_{kp'} +\sum\limits_{p'< k}y^{rji}_{p'k}-\sum\limits _{s:s\neq p}x^{rj}_s\\
 &=y^{rij}_{kp}-x^{rj}_k+\sum\limits_{p'> k:p'\neq p} (y^{rij}_{kp'}-x^{rj}_{p'})+\sum\limits_{p'< k}
 (y^{rji}_{p'k}-x^{rj}_{p'})\\
 &\leq y^{rij}_{kp}-x^{rj}_k\\
 &\leq y^{rij}_{kp}.
 \end{align*}
 The first equality and the first inequality follow from $\eqref{Assignrlt}$. The second equality comes from $\eqref{Assign}$. 
\end{proof}
We also employ the RLT for \eqref{Location} by multiplying it by itself, leading to 
\begin{align}
 \sum\limits_{(i, j)\in V^r \times V^r:i\neq j} y^{rij}_{kp} \le \theta_k, \quad \forall r\in R, \forall k \neq p \in S^2. \label{Locationrlt} 
\end{align}
\begin{remark}
In~\cite{wang2016:relaxations}, we generate RLT inequalities for the location constraints~\eqref{Location} by 
multiplying both sides $x^{rj}_p$ leading to a new bilinear term $x^{rj}_p\theta_k$. Then we introduce new variables 
 $z^{rj}_{pk}$ and inequalities to relax $z^{rj}_{pk} =x^{rj}_p\theta_k$. 
 This leads to a significant number of constraints and new variables but numerically limited improvement. 
 In contrast, constraints~\eqref{Locationrlt} are compact and effective.
\end{remark}

The RLT based formulation can be presented as follows:
\begin{equation}\tag{$\bbP_{\mathrm{RLT}}$}\label{model:RLT}
 \begin{split}
 \min~& \sum\limits_{k\in \mathcal{S}} F_k\theta_k+\sum\limits_{k\in S}A_k\sum\limits_{r\in R}\sum\limits_{i\in V^r}c^{ri}x^{ri}_k+\sum\limits_{e\in E}W_e\phi_e\notag \\
 &\eqref{Assign}, \eqref{Assignrlt}, \eqref{KP1}, \eqref{KP2}, \eqref{QCL}, \notag \\
 & \eqref{Location},\eqref{Locationrlt}, \eqref{linear4},\notag\\
 &\theta_k,\phi_e, x^{ri}_k \in \zeroone,\; r\in R,i\in V^r, k\in S, e\in E. \notag
\end{split}
\end{equation}
We now present three families of strong linear inequalities exploiting the problem structure. 
For each $r\in R, (i,j)\in L^r$, if there is a required throughput between $i$ and
$j$ (i.e., $f^{rij}>0$), then for each link $e \in E$, we have
\begin{equation}\label{Cut1}
\sum\limits_{(k,p)\in S^2:k\neq p, e\in P_{kp} } y^{rij}_{kp}\leq \phi_e. 
\end{equation}
Similarly, for each $r\in R, (k,p)\in S^2: e\in P_{kp}$, if there
exists some throughput mapped along path $P_{kp}$, then for each link $e \in P_{kp}$ we have
\begin{equation}\label{Cut2}
 \sum\limits_{\{i, j\}\in L^r} y^{rij}_{kp} \le \phi_e.
\end{equation}
Another set of valid inequalities can be generated by exploiting some topological property of the substrate graph. 
Recall that a pair of nodes in a graph is connected if there is path between them. 
For each virtual request $r$, we have $\abs{V}^r$ connected VMs mapped to servers. Thus at least $\abs{V}^r$ servers are connected and 
the number of links connecting these servers should be at least $\abs{V}^r-1$. This can be represented as 
 \begin{align}\label{Cut3}
 \sum\limits_{e \in E} \phi_e \ge \max\limits_{r \in R}\{\abs{V^r}\} -1.
 \end{align}

Numerical results in Section~\ref{sec:numerical} show that~\eqref{Cut1}-\eqref{Cut3} strengthen the formulation evidently. 
We now present the compact MIP model for the exact solution procedure of problem $\eqref{bbP}$
as the first contribution of this paper:
\begin{equation}\tag{$\bbP_1$}\label{model:reform}
 \begin{split}
 \min~& \sum\limits_{k\in \mathcal{S}} F_k\theta_k+\sum\limits_{k\in S}A_k\sum\limits_{r\in R}\sum\limits_{i\in V^r}c^{ri}x^{ri}_k+\sum\limits_{e\in E}W_e\phi_e\notag \\
 &\eqref{Assign}, \eqref{Assignrlt}, \eqref{KP1}, \eqref{KP2}, \eqref{QCL}, \notag \\
 & \eqref{Location},\eqref{Locationrlt}, \eqref{Cut1}-\eqref{Cut3} ,\eqref{linear4}, \notag\\
 &\theta_k,\phi_e, x^{ri}_k \in \zeroone,\; r\in R,i\in V^r, k\in S, e\in E. \notag
\end{split}
\end{equation}
\section{The B\&B algorithm}\label{sec:BB}
As indicated in Section~\ref{sec:numerical}, the computational performance of~\eqref{model:reform} is over 10 times more efficient than 
that of~\eqref{model:MC} showing the effectiveness of valid inequalities. However it can be still challenging when the number of VMs is over 30. 

To further improve the scalability we describe a B\&B algorithm to solve the mapping problem to global optimality. 
We first introduce the Lagrange decomposition based lower bounding procedure, then propose an upper bounding heuristic algorithm.
Finally we describe details of the overall algorithm. 
\subsection{The lower bounds}\label{sec:decompose}
The convergence of a B\&B algorithm largely depends on the strength of lower bounds. 
This section aims at generating novel valid inequalities leading to lower bounds that are strong than those provided by the continuous relaxation of~\eqref{model:reform}.
This is achieved by a Lagrange decomposition scheme which involves evaluating a number of subproblems. 
And for this reason we call these inequalities~\emph{Lagrange cuts}. We then generalize the single request decomposition to a decomposition hierarchy.

\subsubsection{Request based decomposition}\label{sec:LD-request}
We first present a decomposition leading to a number of subproblems associated
with each request, each sever and each link. To this end, we disaggregate constraints \eqref{KP1}, \eqref{KP2},
\eqref{QCL} by reformulating them with a handful of auxiliary variables with
corresponding interpretations below
\begin{itemize}
\item {\makebox[2.5cm] { $w^r_k \in [0, C_k]$\hfill } reserved CPU for request $r$ on server $k$,}
\item {\makebox[2.5cm] { $z^r_k \in [0, M_k]$\hfill} reserved Memory for request $r$ on server $k$, }
\item {\makebox[2.5cm] { $\kappa^r_e \in [0, B_e]$\hfill} reserved bandwidth for request $r$ on link $e$. }
\end{itemize}
The equivalent counterparts of \eqref{KP1}, \eqref{KP2} and \eqref{QCL} are then
\begin{alignat}{3}
 &\sum\limits_{i\in V_r} c^{ri}x^{ri}_k &&\le w^r_k , & r\in R, k\in S,&
 \label{eq:ext1}\\
 &\sum\limits_{i\in V_r} m^{ri}x^{ri}_k&&\le z^r_k , & r\in R, k\in S,&
 \label{eq:ext2}\\
 &\sum\limits_{\substack{i,j\in V_r:\\i\neq j}}\sum\limits_{\substack{k,p
 \in S:\\k \neq p,e \in P_{kp}}} f^{rij}y^{rij}_{kp} && \le
 \kappa^r_e , &r\in R, e\in E, & \label{eq:ext3}\\
  & \sum\limits_{r\in R}w^r_k &&\le C_k\theta_k, & k\in S, &
 \hspace{4em} \vlambda \in \real_+^\abs{S} \label{eq:extension1}\\
 & \sum\limits_{r\in R}z^r_k &&\le M_k\theta_k, & k\in S,
 &\hspace{4em} \vmu \in \real_+^\abs{S}\label{eq:extension2}\\
 & \sum\limits_{r\in R}\kappa^r_e && \leq B_e\phi_e, & e\in E.
 & \hspace{4em} \vsigma \in \real_+^\abs{E}\label{eq:extension3}
\end{alignat}
To make the problem separable by request while ensuring strong lower bounds, we copy variables $\vtheta$ and
$\vphi$ by introducing the following constraints:
\begin{alignat}{5}
 & \theta_k^r \le \theta_k, && r\in R, k\in S,
 &&\veta\in\real^{\abs{R}\times\abs{S}}_+ \label{eq:extension4},\\
 & \phi_e^r \le \phi_e, \hspace{4em}&&e\in
 E,\hspace{4em}&&\vzeta\in\real^{\abs{R}\times\abs{E}}_+ \label{eq:extension5}.
\end{alignat}
\eqref{eq:extension4}-\eqref{eq:extension5} imply the fact that if a server/link is used by one request then it must be on and 
conversely if a sever/link is on then it is not necessarily used by all the requests. As a result, the upper bounds of $w^r_k, z^r_k, \kappa^r_e$ 
can be strengthened to $C_k\theta^r, M_k \theta^r_k$ and $B_e\phi^r_k$ respectively. In addition we relax the connectivity constraint~\eqref{Cut3} 
with $\rho \in \real_+$. 

We replace $\theta_k$ with $\theta_k^r$ in constraints
\eqref{Location} and $\phi_e$ with $\phi_e^r$ in constraints \eqref{QCL}. 
Let us denote the resulting formulation as a~\emph{lifted} version of~\eqref{model:reform}:
\begin{align}\label{model:reform2}
\bbP_{\mathrm{2}} : \{\eqref{model:reform}\} \cap \bigcurly{\eqref{eq:extension1},
 \eqref{eq:extension2}, \eqref{eq:extension3}, \eqref{eq:extension4}, \eqref{eq:extension5}}.
\end{align}
\begin{prop}\label{lemma:projection}
The projection of the feasible region of \eqref{model:reform2} in variables $(\vx,\vy,\vtheta,\vphi)$ is exactly the feasible region of \eqref{model:reform}.
\end{prop}
\begin{proof}
For any feasible point $\pv = (\vw, \vz, \vkappa, \vtheta^r, \vphi^r,\vx,\vy, \vtheta, \vphi)$ in~\eqref{model:reform2}, one can get a feasible point in~\eqref{model:reform} by truncating the components of variables $(\vx,\vy,\vtheta,\vphi)$
from $\vv$. 
\end{proof}

Relaxing the five sets of constraints with associated Lagrange multipliers
$\vlambda, \vmu, \vsigma, \veta, \vzeta, \rho$ leads to the Lagrange function over variables
$\left(\vlambda,\vmu,\vsigma, \veta, \vzeta, \rho; \vw, \vz, \vkappa,
\vx,  \vtheta,\vphi \right)$. For ease of notation, let $\pv = (\vw, \vz, \vkappa,\vx,\vtheta, \vphi)\in \real^p_+, 
\dv = (\vlambda,\vmu,\vsigma, \veta, \vzeta, \rho) \in \real^{d}_+$, where $p$ and $d$ represent the respective number of 
variables in primal and dual space. The Lagrange is then
\begin{equation}\label{eq:lagrange}
 \begin{split}
\La\left(\dv, \pv \right)= &\sum\limits_{r\in
 R}\left(\sum\limits_{k\in S} \left(A_k\sum\limits_{i\in
 V_r}c^{ri}x^{ri}_k+w^r_k\lambda_k+\mu_kz^r_k + \theta^r_k\eta^r_k \right)+\sum\limits_{e\in E}
 \left(\sigma_e\kappa^r_e + \zeta^r_e \phi^r_e \right)\right)\\ 
   & + \sum\limits_{k}
 \left(F_k- C_k\lambda_k- M_k\mu_k - \sum\limits_{r\in R} \eta^r_k \right)\theta_k\\
 & + \sum\limits_{e} \left(W_e- B_e\sigma_e - \sum\limits_{r \in R} \zeta^r_e - \rho
 \right)\phi_e \\
 & + (\max\limits_{r\in R}\abs{V^r} - 1) \rho.
 \end{split}
\end{equation}
For the sake of clarity let $ \tau = (\max\limits_{r\in R}\abs{V^r} - 1) \rho$.
For any $\vv^* \in \real^d_+$, it holds that the infimum 
\begin{align}\label{eq:dualeval}
\Dual(\vv^*) = \inf\limits_{\vv \in \X}\La(\vv^*, \vv)
\end{align}
provides a lower bound of the optimal value of the mapping problem, where $\X$ represents the remaining constraint set. 
Since $\X$ is compact, the supreme is attainable.
Given $\vv^*$, the Lagrange function is separable w.r.t. each request, server and edge. 
Thus evaluating $\min \limits_{\vv}\La(\vv^*, \vv)$ reduces to evaluating $\abs{R}+\abs{S}+\abs{E}$ subproblems.

For each $r\in R$, we evaluate $\Dual^r(\dv)$ by solving
\begin{align}
 \min \hspace{2em}& \sum\limits_{k\in S}\left(A_k\sum\limits_{i\in
 V_r}x^{ri}_kc^{ri}+\lambda_kw^r_k+\mu_kz^r_k + \theta^r_k\eta^r_k
\right)+\sum\limits_{e\in E} \left(\sigma_e\kappa^r_e + \zeta^r_e \phi^r_e
\right) \tag{\textrm{Sub$_r$}} \label{eq:subr}\\
 \subto\; &\eqref{Assign}, \eqref{Assignrlt}, \eqref{QCL}, 
  \eqref{Location},\eqref{Locationrlt}, \eqref{Cut1}-\eqref{Cut3} ,\eqref{linear4},
 \eqref{eq:ext1},\eqref{eq:ext2},\eqref{eq:ext3}, \notag\\ 
 &w^r_k \le C_k \theta^r_k, ~\forall k \in S \label{eq:cpur} \\
 &z^r_k \le M_k \theta^r_k, ~\forall k \in S \label{eq:memoryr}\\
  &\kappa^r_e \le B_e \phi^r_e, ~\forall e \in E, \label{eq:bandwidthr}\\
   & x^{ri}_k, \theta_k^r, \phi_e^r,  \in \{0, 1\} ~\forall i, k, e. \notag
 \end{align}
As server $k$ is used by request $r$ if and only if there is a VM mapped to server $k$, We can strengthen~\eqref{Location} in~\eqref{eq:subr} as follows
\begin{align}\label{eq:validsubr}
&\sum\limits_{i \in V^r} x^{ri}_k = \theta_k^r.
\end{align}
Moreover, given that each request graph $G^r$ is connected, constraints~\eqref{Cut3} in~\eqref{eq:subr} can be replaced with
\begin{subequations}\label{eq:Cut3r}
\begin{align}
&\sum\limits_{e \in E} \phi_e^r \ge \sum\limits_{k \in S} \theta_k^r -1,  \\
 \theta_k^r & \le \sum\limits_{e \in E: k \in e} \phi_e^r , \; \; \; k\in S, 
\end{align}
\end{subequations}
as the graph induced by servers that are used by request $r$ is also connected.

In what follows, let $X^r$ be the feasible region of~\eqref{eq:subr} in primal variables $(\vw^r, \vz^r, \vkappa^r,\vx^r, \vtheta^r, \vphi^r)$, i.e., 
\begin{align*}
X^r = \large\{&(\vw^r, \vz^r, \vkappa^r,\vx^r, \vtheta^r, \vphi^r): \eqref{Assign}, \eqref{Assignrlt}, \eqref{QCL},
  \eqref{Locationrlt}, \eqref{Cut1}-\eqref{Cut2} ,\eqref{linear4},\\
  &\eqref{eq:ext1},\eqref{eq:ext2},\eqref{eq:ext3}, \eqref{eq:validsubr}, \eqref{eq:Cut3r}, \eqref{eq:cpur}-\eqref{eq:bandwidthr}, 
   x^r_{ik}, \theta^r_k, \phi^r_e \in \{0, 1\}\large\}.
\end{align*}

Similarly, for each $k\in S$, and $e\in E$, we evaluate $\Dual^k(\vv^*)$ and $\Dual^e(\vv^*)$ via
 \begin{align}
 \Dual^k(\vlambda, \vmu, \veta) &= \min\limits_{\theta_k \in \cube} \left(F_k-
 C_k\lambda_k- M_k\mu_k -  \sum\limits_{r\in R} \eta^r_k \right)\theta_k, \label{eq:subserver}\\
 \Dual^e (\vsigma, \vzeta)&= \min\limits_{\phi_e \in \cube} \left(W_e- B_e\sigma_e -
 \sum\limits_{r \in R} \zeta^r_e -\rho \right)\phi_e.\label{eq:sublink}
\end{align}
It is straightforward to see that the optimal solution of the above problem is $1$ if the coefficient is negative and $0$ if nonnegative.
Consequently, the dual objective denoted by $\Dual(\dv)$ is additively composed of functions $\{\Dual^r\}_{r \in R}$, 
$\{\Dual^k\}_{k \in S}$ and $\{\Dual^e\}_{e \in E}$ as follows
\begin{equation}\label{eq:dualbound} 
 \Dual(\dv)= \sum\limits_{r\in R} \Dual^r(\dv)+\sum\limits_{k\in S}
 \Dual^k(\dv) +\sum\limits_{e\in E} \Dual^e (\dv) +  \tau
 \end{equation}
 Let us assume that primal problem~\eqref{model:reform2} is strictly feasible. Then the Slater's condition holds and the dual problem has a nonempty 
 compact set of maximum points~\cite{lemarechal2001}.  Thus the dual problem can be defined below
 \begin{equation}\label{eq:dual}
 \max\limits_{\dv \in \real^d_+} \Dual(\dv). 
 \end{equation} 
 By weak duality, it holds that 
 \begin{align*}
 \max\limits_{\dv \in \real^d_+} \Dual(\dv) \le z^*
 \end{align*}
 where $z^*$ is the optimum of~\eqref{model:reform}. 
 
 \subsubsection{On the strength of lower bounds}\label{sec:lb}
 It is known that the strong duality generally does not hold between~\eqref{model:reform} and~\eqref{eq:dual}.
 In other words, the duality gap exists. Naturally we may ask: 
 \begin{quotation}
 \textit{Could we get a priori knowledge or intuition on the quality of the lower bound without actual numerical experiments?}
 \end{quotation}
Much effort has been devoted to relevant investigations in  different contexts (see~e.g.~\cite{geoffrion1974,guignard1987,lemarechal2001,Lemarechal2001:gap}).
Among them, we recall the following theorem.
 \begin{thm}\cite{geoffrion1974}\label{eq:lrbound}
 Consider a mixed integer linear problem expressed as
 \begin{align*}
 \min\limits_{x} \bigcurly{cx :  ~Ax \le b, x \in X}.
 \end{align*}
 where $X$ is bounded and compact.  Relaxing constraints $Ax \le b$ with Lagrange multipliers $\lambda \ge 0$ leads to 
 \begin{align*}
 g(\lambda) = \min\limits_{x \in X}~cx+ \lambda^T(Ax -b).
 \end{align*}
 Then it holds that 
 \begin{align*}
 \max\limits_{\lambda \ge 0}~ g(\lambda) = \min\bigcurly{cx: ~Ax \le b, x\in \conv(X)}
  \end{align*}
  where $\conv (X)$ represents the convex hull of $X$.
 \end{thm}
Later Lemar\'echal generalized the above result for general Mixed Integer Nonlinear Problems~(MINLPs).~We refer interested readers to~\cite{lemarechal2001,Lemarechal2001:gap} 
for details. 
\begin{cor}\label{cor:quality}
 The optimum of~\eqref{eq:dual} dominates (greater than or equal to) that of continuous relaxation of
 \eqref{model:reform2} and it amounts to outer-approximating the convex hull of the primal feasible region with the following
 set
 \begin{align}\label{cor:primalcharacter}
 \Sy = \conv \{\bigtimes_{r\in R} X^r\} \bigcap \bigcurly{\vv: \eqref{Cut3},\eqref{eq:extension1}, 
 \eqref{eq:extension2}, \eqref{eq:extension3},
 \eqref{eq:extension4},\eqref{eq:extension5}}, 
 \end{align}
 where $X^r$ denote the feasible region of subproblem~\eqref{eq:subr} and $\bigtimes$ denotes the Cartesian product operation: $\bigtimes_{r\in R} X^r = \bigcurly{(x_1, \dots, x_{\abs{R}}):
 x_i \in X^r, \; r \in R}$.
\end{cor}
\begin{proof}
It follows directly from Theorem~\ref{eq:lrbound} that the Lagrange decomposition amounts to constructing $\Sy$. Since $\conv X^r$ is the tightest convex relaxation of $X^r$, the optimum of~\eqref{eq:dual} is greater than 
or equal to the linear relaxation objective value. 
\end{proof}
\subsubsection{The choice of Lagrange multipliers} 
Even though each subproblem~\eqref{eq:subr} is easier than
problem~\eqref{model:reform}, they are still computationally costly; subgradient
algorithms~\cite{subgradient} 
often take hundreds of iterations to converge. To reduce the number of iterations evidently, we propose to
first solve the continuous relaxation of the extended
formulation~\eqref{model:reform2} to optimality and take the dual multipliers of
constraints~\eqref{eq:extension1}- \eqref{eq:extension5} 
as the initial values the Lagrange multipliers. With these initial values, we evaluate~\eqref{eq:dualbound} and obtain $l^1$ as the optimal value. 
Let $l^{cts}$ be optimal value of the continuous relaxation of~\eqref{model:reform2}. 
\begin{prop}\label{prop:lower}
It holds that $l^{1} \ge l^{cts}. $ 
\end{prop}
\begin{proof}
Let $\bar \vv^*$ be the corresponding dual values w.r.t to constraints~\eqref{eq:extension1}- \eqref{eq:extension5} and $\overline \Dual^r, \overline \Dual^k, \overline \Dual^e$ be the respective optimal value of the continuous relaxation of~subproblem~\eqref{eq:subr},~\eqref{eq:subserver},~\eqref{eq:sublink} associated with request $r \in R$ with the Lagrange multipliers $\bar \vv^*$.
\begin{enumerate} 
\item  Due to integrality constraints in~\eqref{eq:subr}, it holds that $\overline \Dual^r \le \Dual^r$ for each $r \in R$; thus $l^1 \ge \sum\limits_{r\in R} \overline \Dual^r(\bar \dv)+\sum\limits_{k\in S}
 \overline \Dual^k(\bar \dv) +\sum\limits_{e\in E} \overline \Dual^e (\bar \dv) + \tau$.
\item Let $\bar \vv$ be the optimal solution of the continuous relaxation of~\eqref{model:reform2}, then 
$\sum\limits_{r\in R} \overline \Dual^r(\bar \dv)+\sum\limits_{k\in S} \overline \Dual^k(\bar \dv) +\sum\limits_{e\in E} \overline \Dual^e (\bar \dv) $
 is the dual optimal value of the continuous relaxation of~\eqref{model:reform2}. By strong duality we have $l^{cst} = \sum\limits_{r\in R} \overline \Dual^r(\bar \dv)+\sum\limits_{k\in S}
 \overline \Dual^k(\bar \dv) +\sum\limits_{e\in E} \overline \Dual^e (\bar \dv) + \tau$.
\end{enumerate}
Combining the above arguments leads to $l^1 \ge l^{cst}$.
\end{proof}
\subsubsection{Lagrange cuts}\label{sec:lagrangecuts}
In this section we show that we can generate some valid inequalities to capture the strength of the Lagrange
decomposition characterized by \eqref{cor:primalcharacter}. Thus we call these inequalities~\emph{Lagrange cuts}. 
As a result we can strengthen the linear relaxation of~\eqref{model:reform2} upon each resolution
of~\eqref{eq:subr}.
\begin{prop}
For each $r\in R$,~\eqref{eq:subr} amounts to finding a supporting hyperplane of $\conv(X^r)$ with outer normal vector 
$(-\vlambda, -\vmu,  -\vsigma, -\{A_kc^{ri}\}_{i,k}, -\veta, -\vzeta)$ 
defined by equation 
\begin{align*}
H^r(\vv) = \Dual^r(\vv^*) - \sum\limits_{k\in S}(A_k\sum\limits_{i\in V_r}x^{ri}_kc^{ri}+\lambda_kw^r_k+\mu_kz^r_k + \theta^r_k\eta^r_k )-\sum\limits_{e\in E} (\sigma_e\kappa^r_e + \zeta^r_e \phi^r_e) = 0
\end{align*}
where $\vv = (\vw^r, \vz^r, \vkappa^r,\vx^r, \vtheta^r, \vphi^r)$ and $\Dual^r(\vv^*)$ is the optimal value of~\eqref{eq:subr}. 
\end{prop}
\begin{proof}
By the definition of~\eqref{eq:subr}, it holds that 
\begin{align}\label{eq:lagrangecut}
H^r(\vv) \le 0
\end{align}
 for any point $\vv \in X^r$ and it 
exists at least one point $\vv' \in X^r$ such that $H^r(\vv') = 0$, ending the proof. 
\end{proof}
The proposition above provides  a non-trivial valid inequality $H^r(\vv) \le 0$ for each $\vv \in \conv X^r$ and therefore 
it is also valid for the convex set defined in $\eqref{cor:primalcharacter}$. Thus we can append these inequalities to strengthen 
the linear relaxation of~\eqref{model:reform2}. Moreover we show that the resulting linear relaxation value of~\eqref{model:reform2} 
will be greater than or equal to the current Lagrange lower bound. 
\begin{prop}\label{prop:cut}
The optimal value of the continuous relaxation of~\eqref{model:reform2}
    augmented by $H^r(\vv) \le 0 \; (r\in R)$ is greater than or equal to $l^{1}$.
\end{prop}
\begin{proof}
Let $\pi^r = \sum\limits_{k\in S}(A_k\sum\limits_{i\in
    V_r}x^{ri}_kc^{ri}+\lambda_kw^r_k+\mu_kz^r_k + \theta^r_k\eta^r_k
    )+\sum\limits_{e\in E} (\sigma_e\kappa^r_e + \zeta^r_e \phi^r_e) $. 
    Then aggregating Lagrange cuts~\eqref{eq:lagrangecut} over $r \in R$ leads to 
    \begin{align*}
&\sum\limits_{r \in R} \pi^r \ge \sum\limits_{r \in R} \Dual^r(\vv^*) 
\end{align*}
which is equivalent to  
\begin{align}\label{eq:target}
&\sum\limits_{r \in R} \pi^r  + b
\ge \sum\limits_{r \in R} \Dual^r(\vv^*) + b,
\end{align} 
where $b = \sum\limits_{k \in S}(F_k-
 C_k\lambda_k- M_k\mu_k -  \sum\limits_{r\in R} \eta^r_k)\theta_k + \sum\limits_{e \in E} 
 (W_e- B_e\sigma_e - \sum\limits_{r \in R} \zeta^r_e-\rho) \phi_e + \tau$.
 On the one hand, \eqref{eq:subserver} and~\eqref{eq:sublink} imply that 
\begin{align} \label{eq:arg1} 
&\sum\limits_{r \in R}\Dual^r(\vv^*) + b \ge \sum\limits_{r \in R}\Dual^r(\vv^*) +\sum\limits_{k\in S}  \Dual^k(\dv) +\sum\limits_{e\in E} \Dual^e (\dv) + \tau = l^1. 
\end{align} 
  On the other hand, we note that  $\sum\limits_{r \in R} \pi^r  + b \le f(\vv)$ is exactly the Lagrange function~\eqref{eq:lagrange}. 
  Thus by weak duality it holds that 
    $\sum\limits_{r \in R} \pi^r  + b \le f(\vv)$ for any solution satisfying~\eqref{KP1}--\eqref{QCL}. 
     The proof is then complete.
\end{proof}
\begin{remark}
 The above discussion reveals the fact that Lagrange decomposition procedure can be regarded as a (Lagrange) cut generation procedure which can be called as needed 
in our B\&B algorithm. 
\end{remark}
{
}
\subsubsection{A generalized decomposition hierarchy} \label{sec:hierarchy}
Corollary~\ref{cor:quality} also indicates a hierarchy of request based Lagrange
relaxations. Specifically, we may consider any possible partition of the request set $R$ and apply
the aforementioned decomposition strategy to the partition. 

For ease of presentation, let us assume that we partition set $R$ into 
$m \in \{1, \dots, \abs{R}\}$ disjoint subsets and denote this partition $\mathcal{P}_m = \{I_1, I_2, \dots, I_m\}$, where $I_i \subset R, ~\forall i = 1,
\dots, m$. If $m = \abs{R}$, then we get the single request based decomposition and 
if $m=1$, we end up with the lifted formulation~\eqref{model:reform2}. 
Accordingly each subproblem~\eqref{eq:subr} is defined on subset $I_i$. For
instance constraint $\eqref{eq:ext1}$ becomes 
\begin{align*}
 &\sum\limits_{r \in I_j}\sum\limits_{i\in V_r} c^{ri}x^{ri}_k \le w^{j}_k,\; j\in \{1, \dots, m\}, k\in S.
 \label{eq:extnew}
\end{align*}

This leads to a hierarchy of decompositions and we denote by $l_m$ the resulting lower bound associated with parameter $m$.
\begin{cor} \label{cor:bounds}
Let $m_1, m_2$ be any integer in $\bigcurly{1, \dots, \abs{R}}$. If $m_1 \le m_2$ and for any subset $I \in \mathcal{P}_{m_2}$ there exists a subset $I' \in \mathcal{P}_{m_1}$ such 
that $I \subset I'$, then $l_{m_1} \ge l_{m_2}$. 
\end{cor}
\begin{proof}
By Corollary~\ref{cor:quality}, the Lagrange decomposition amounts to the following outer-approximation
 \begin{align*}
 S^m = \conv \{\bigtimes_{i \in \bigcurly{1, \dots, m}} X^i\} \bigcap \bigcurly{\vv: \eqref{Cut3},\eqref{eq:extension1},\eqref{eq:extension2}, \eqref{eq:extension3},
 \eqref{eq:extension4},\eqref{eq:extension5} }, 
 \end{align*}
 where $X^i$ represents the feasible region of constraints associated with requests in set $I_i$.
 The condition that for any subset $I \in \mathcal{P}_{m_2}$ there exists a subset $I' \in \mathcal{P}_{m_1}$ such 
that $I \subset I'$ implies that $S^{m_1} \subseteq S^{m_2}$ ending the proof.
\end{proof}
As an illustration, we can partition 6 virtual requests (30 VMs) with $R =\{1, \dots, 6\}$ to $6$ subsets, each has a single request; 
$3$ subsets with $I_1 = \{1, 2\}, I_2 = \{3, 4\}, I_3=\{5, 6\}$.
Let $l_6, l_3$ be 
their respective Lagrange lower bounds associated with the partition-based Lagrange relaxation. Then it follows from Corollary~\ref{cor:bounds} that $l_6 \le l_3$.
To accelerate the solution procedure of subproblems involving multiple requests, one can apply
the aforementioned request-based decomposition approach recursively. 

\subsection{The upper bounds}\label{sec:primal}
It is known that solving the dual problem~\eqref{eq:dual} (or evaluating~\eqref{eq:dualbound}) does not produce any
primal feasible solution. 
Recovering a high quality one often calls for appropriate heuristics. 
In this paper we exploit the information of subproblems~\eqref{eq:subr} and
construct a feasible solution and an upper bound. If the resulting upper bound 
is weak we improve the solution using the~\emph{local branching} technique
introduced in~\cite{fischetti03:branching}.

The overall algorithm is presented in Alg.~\ref{alg:heuristic}. Given an optimal
solution to the Lagrange problem~\eqref{eq:lagrange}, we partition the index set $S$ 
into two disjoint subsets by inspecting values of $\sum\limits_{r \in
R} \overline \theta^r_k$. 
On the one hand, if $\sum\limits_{r \in R} \theta^r_k = 0$, we guess
that server $k$ is unlikely used in a good feasible solution and thus set
$\theta_k = 0$. This eliminates binary $x^{ri}_k~(r \in R~ i \in V_r)$ by location
constraint~\eqref{Location}. Moreover we restrict this unused server $k$ isolated from used ones by imposing
$\phi_e =  0 ~~\forall e\in E:  k \in e$. 
On the other hand, if $\sum\limits_{r \in R} \theta^r_k \ge n \; (n \in \bigcurly{1, \dots, \abs{R}})$ we set $\theta_k=1$; 
to further reduce the number of binary variables, we fix binary variables $\{x^{ri}_k: r \in R, i \in V^r\}$ if they do not violate coupling constraints~\eqref{eq:extension1}--\eqref{eq:extension2}.

As a result we end up with a small MIP which may return a feasible solution. If the MIP is infeasible we switch on the cheapest server among the unused servers and repeat the procedure.  

Let $LB$ be the lower bound obtained by solving~\eqref{eq:lagrange} and $UB,
\tilde \vv = (\tilde \vx, \tilde \vtheta, \tilde \vphi) $ the resulting
upper bound and feasible solution using~Alg.~\ref{alg:heuristic}. If $\textrm{gap} = \frac{UB-LB}{UB}$
is large we mploy the local branching technique~\cite{fischetti03:branching} to
improve it by appending the distance constraint~\eqref{eq:localbranch}
to~\eqref{model:reform} 
\begin{equation} \label{eq:localbranch}
    d(\vv, \tilde\vv) \le \pi.  
\end{equation}
where $d(\vv, \tilde \vv)$ represents the
distance between the optimal solution $\vv$ and the current feasible solution
$\tilde \vv$ and $\pi$ is an integer parameter. Following~\cite{fischetti03:branching}, we set $d(\vv, \tilde\vv) = \sum\limits_{k
\in S: \tilde \theta_k = 1} (1 - \theta_k) +  \sum\limits_{e
\in E: \tilde \phi_e = 1} (1 - \phi_e) + \sum\limits_{r \in R} \sum\limits_{i
\in V^r}\sum\limits_{k \in S: \tilde x^{ri}_k = 1} (1 - x^{ri}_k) $ and choose $\pi \in \{10, \dots, 20\} $.
We then solve the resulting augmented problem and try to get a better solution.

\begin{algorithm}[h!]
\caption{Repairing heuristic}
 \label{alg:heuristic}
 \SetAlgoLined
\SetKwInOut{Input}{Input}
\SetKwInOut{Output}{Output}
    \Input{Solution $\bar \vv$ to
    problem~\eqref{eq:lagrange}; an integer $n\in \{1,\abs{R}$\}}
    \Output{An feasible solution to problem~\eqref{model:reform} and upper bound
    $UB$.}
    Let $\Theta_1 = \{k \in S: \sum_{r \in R} \overline \theta_k^r  \ge n \}$ and
    $\Theta_0 = \{k \in S: \sum_{r \in R}  \overline \theta_k^r = 0\}$.\\
    \While{$\abs{\Theta_0} \ge 0$}{
    Let $\theta_k = 0~\forall k \in \Theta_0$. \\
   For each $k \in \Theta_0$, impose $\phi_e = 0, ~\forall e\in E: k \in e$. \\
    \ForEach{$k \in \Theta_1$} {
        \If{$C_k - \sum\limits_{r \in R} w^r_k \ge  0$ and $M_k - \sum\limits_{r \in R} z^r_k \ge 0$}
        { Let $\theta_k = 1, ~x^{ri}_k = \overline x^{ri}_k \; (\forall r \in R \;  i \in V^r)$.
       }
    }
    Solve model~\eqref{model:reform} with partially fixed $\vtheta, \vx$ and constraints~\eqref{eq:strongerCut3}. \\
   \eIf{infeasible}{Find $k = \argmin\{F_k: k \in \Theta_0\}$.  \\
    Let $\Theta_0 = \Theta_0 \setminus \{k\}$.
    }
    {break the loop.}
}
\If{$\textrm{UB-LB}\ge \epsilon\textrm{UB}$}{
Append~\eqref{eq:localbranch} to~\eqref{model:reform}, which is solved to update the feasible solution.
}
\end{algorithm}
\begin{remark}
To accelerate the procedure of Alg.~\ref{alg:heuristic} we restrict the MIP solution procedure within $3 \times \abs{R}$ seconds. Of-course we do not pretend that 
it guarantees a high quality feasible solution. For this reason, Alg.~\ref{alg:heuristic} is called several times in our B\&B algorithm. 

\end{remark}

\subsection{The algorithm}
In this section, we elaborate details of the overall branch-and-bound algorithm. As highlighted before,  the algorithm uses Lagrange decomposition stated in
Sec.~\ref{sec:decompose} and Alg.~\ref{alg:heuristic} to generate strong lower
and upper bounds. In addition it distinguishes from the standard B\&B algorithm 
in terms of branching rules and dynamic cut generation.

\subsubsection{Branching} \label{sec:branching}
The B\&B bound algorithm focus on branching over $\vtheta$ and $\vphi$ variables.  If
a node is not fathomed when all $\vtheta, \vphi$ are integral we export the arising
subproblem to standard MIP solvers. This is due to the following reasons.  First, cost coefficients of $\vtheta$ and $\vphi$ are usually much larger than those of $\vx$. 
And as will be explained in Section~\ref{sec:cuts} a number of strong valid inequalities can be generated on the fly by fixing values of $\vtheta$ or $\vphi$.  
Thus one can usually fathom a node by just branching
over $\vtheta$ and $\vphi$. Second, there might be many symmetries among variables $\vx$, where symmetry breaking techniques should be investigated.
This however is out of the scope of this paper. Instead we exploit relevant features of existing MIP solvers. 

For $\vtheta, \vphi$, we select variables that are most inconsistent with
respect to linking constraints~\eqref{eq:extension4}
and~\eqref{eq:extension5} as it is more likely to fathom a node.
Let us illustrate this point for variables $\vtheta$ and denote by $\bar \theta_k^r, \bar \theta_k$ the respective optimal values of~\eqref{eq:subr} and~\eqref{eq:subserver} for each $k\in S$. 
If $\sum\limits_{r \in R}\bar \theta^r_k = 0$, we will not branch it as the constraint 
$\theta^r_k \le \theta_k$ always holds. Otherwise if $\sum\limits_{r \in R}
\bar \theta^r_k \ge 1$ and $\bar \theta_k = 0$, we branch on variable $\theta^k$ and create two nodes with
$\theta_k = 0$ or $\theta_k = 1$. Then one can try to improve lower bounds of both nodes by solving 
a continuous relaxation or a Lagrange decomposition procedure. 
Note that the Lagrange bounds of these two child nodes can be improved even~\emph{without} updating the Lagrange multipliers.  
For the child node with $\theta_k = 0$, we can improve the Lagrange lower bound 
by simply evaluating~\eqref{eq:subserver} indexed by $k$ and~\eqref{eq:subr} where $\theta^r_k$ was determined as $1$ at its parent node;
 while for the child node with~$\theta_k = 1$, we can improve the lower bound by evaluating~\eqref{eq:subserver} indexed by $k$.
When all $\vtheta$ are
consistent with $\vtheta^r$, we branch over the most fractional component. 
\begin{remark}
Our numerical experiments show that without updating the Lagrange multipliers, the improvement of lower bound is significant for the child node with $\theta_k = 0$ 
but quite small for the node with $\theta_k=1$.  Thus for the latter case, we may still update Lagrange multipliers (by solving the continuous relaxation problem~\eqref{model:reform2}).
\end{remark}
\subsubsection{Bounding }
As highlighted, variables $\vtheta$ and $\vphi$ have dominating
coefficients in the objective function. Thus bounding these variables might be
beneficial in the procedure of B\&B algorithm.   

Let $\F_{\eqref{model:reform2}}$ be the continuous relaxation of the feasible
region of~\eqref{model:reform2}. We evaluate bounds of $\sum\limits_{k
\in S} \theta_k, \sum\limits_{e \in E} \phi_e$ by solving linear programs over $
\F_{\eqref{model:reform2}}$ augmented by Lagrange cuts and the upper bounding
constraint $f(\vv) \le UB$. For instance the upper bound of  $\sum\limits_{k
\in S} \theta_k$ is obtained by solving the following linear problem
\begin{align}\label{eq:bounding}
p = \max\bigcurly{\sum\limits_{k\in S} \theta_k: \subto \; \vv \in \F_\eqref{model:reform2}, f(\vv) \le UB, \eqref{eq:lagrangecut}}
\end{align}
and then taking $u_\theta = \floor{p}$. Similarly we can get the upper bound of $\sum\limits_{e \in E} \phi_e$ which we denote by $u_\phi$.
Moreover one can check the connectivity of used servers using the following proposition. 
\begin{prop}
Given that all request graphs $G^r \; (r \in R)$ are connected, all used servers are connected if either of the following inequality holds
\begin{subequations}\label{eq:boundchecking}
\begin{align}
u_\theta \le \min\limits_{r^1, r^2 \in R}\bigcurly{\abs{V}^{r_1} + \abs{V}^{r_2}}-1, \\
u_\phi \le \min\limits_{r^1, r^2 \in R}\bigcurly{\abs{V}^{r_1} + \abs{V}^{r_2}}-3. \; 
\end{align}
\end{subequations}
\end{prop}
\begin{proof}
The first inequality implies that any pair of two virtual requests share at least one server. 
The second one indicates that that at least one link is shared by any pair of two requests.
Indeed suppose that there exist a pair of two requests $r_1, r_2 \in R$ using two disjoint sets of links.  
Then we have that 
\begin{align*}
u_\phi  \ge \sum\limits_{e \in E} \phi^{r_1}_e + \phi^{r_2}_e 
\end{align*}
On the other hand~\eqref{eq:Cut3r} implies that 
\begin{align*}
\sum\limits_{e \in E} (\phi^{r_1}_e + \phi^{r_2}_e )  \ge \sum\limits_{k \in S} (\theta^{r_1}_k + \theta^{r_2}_k)-2 
\end{align*}
And constraints~\eqref{eq:validsubr} and~\eqref{Assign} imply that 
\begin{align*}
\sum\limits_{k \in S} (\theta^{r_1}_k + \theta^{r_2}_k) = \sum\limits_{k \in S} \left(\sum\limits_{i \in V^{r_1}} x^{r_1 i}_k +\sum\limits_{i \in V^{r_2}} x^{r_2 i}_k\right) = \abs{V}^{r_1} + \abs{V}^{r_2}
\end{align*}
These lead to that 
\begin{align*}
u_\phi  \ge \bigcurly{\abs{V}^{r_1} + \abs{V}^{r_2}}  - 2
\end{align*}
which contradicts the inequality $u_\phi \le \min\limits_{r^1, r^2 \in R}\bigcurly{\abs{V}^{r_1} + \abs{V}^{r_2}}-3$. 
\end{proof}
Note that the fact that the graph induced by used servers is connected will be exploited in the next section to derive some strong cuts. 
\subsection{Cuts}\label{sec:cuts}
In our B\&B algorithm, two sets of cuts are added to the continuous relaxation of~\eqref{model:reform2} at each child node dynamically, namely, the Lagrange cuts~\eqref{eq:lagrangecut} 
and~\emph{connectivity cuts}. The former has been elaborated in Subsection~\ref{sec:lagrangecuts} and we now introduce the latter.

For ease of presentation, we denote by $H'$ the subgraph of $H$ induced by all used servers and used links. If $H'$ is connected, then 
the following inequalities are valid 
\begin{subequations}
\begin{align}
 \sum\limits_{e \in E} \phi_e &\ge \sum\limits_{k \in S} \theta_k -1,  \label{eq:strongerCut3}\\
 \theta_k & \le \sum\limits_{e \in E: k \in e} \phi_e,  \; k\in S,\label{eq:strongerCut3_2}
\end{align}
\end{subequations}
where the first one stipulates that at least $\sum\limits_{k \in S} \theta_k -1$ links are used and the second one 
comes from the definition of connectivity. 
It is straightforward to see that~\eqref{eq:strongerCut3} dominates~\eqref{Cut3}. Let us call~\eqref{eq:strongerCut3} 
and~\eqref{eq:strongerCut3_2} connectivity cuts. 

The overall B\&B procedure is presented in Algorithm~\ref{alg:global}.

{\scriptsize
\begin{algorithm}
\caption{{The B\&B algorithm}}
 \label{alg:global}
 \SetAlgoLined
\SetKwInOut{Input}{input}
\SetKwInOut{Output}{output}
\Input{an instance data, the optimality tolerance $\epsilon$, upper bound tolerance $\delta$}
\Output{An optimal solution and the optimal objective value}
    {\bf Step 1: Initialization}\\
  At root node solve the continuous relaxation of the reformulation~\eqref{model:reform2}. \\ 
  Extract the dual multipliers with respect to constraints~\eqref{eq:extension1}-\eqref{eq:extension5}. \\
   For each $r \in R$~ evaluate~\eqref{eq:subr}; evaluate\eqref{eq:subserver}, \eqref{eq:sublink}. \\
  Initialize lower bound $LB$ with~\eqref{eq:dualbound} and store $\abs{R}$ Lagrange cuts~\eqref{eq:lagrangecut}.  \\
Initialize $UB$  using  Alg.~\ref{alg:heuristic} and go to step 4. \\
{\bf Step 2: Evaluation}\\
\eIf{the branched variable (e.g.~$\theta_k$) is fixed to be $0$ and $\sum\limits_{r \in R} \bar \theta^r_k >0$}{
For each $r \in R $ such that $\theta_k^r = 1$~ evaluate~\eqref{eq:subr}; evaluate \eqref{eq:subserver} and~\eqref{eq:sublink}.
}{
Solve the continuous relaxation of the reformulationcenhanced by~\eqref{eq:lagrangecut}\\ 
  Extract the dual multipliers with respect to constraints~\eqref{eq:extension1}-\eqref{eq:extension5}. \\
   For each $r \in R$~ evaluate~\eqref{eq:subr}; evaluate\eqref{eq:subserver},~\eqref{eq:sublink}. \\
Update lower bound $LB$ with~\eqref{eq:dualbound} and store $\abs{R}$ Lagrange cuts~\eqref{eq:lagrangecut}.}
If {$UB - LB \ge \delta UB$,  then update $UB$  using  Alg.~\ref{alg:heuristic}.} Go to Step 3.\\
{\bf Step 3: Termination} \\
A node is fathomed if one of the following conditions is met: 
\begin{enumerate}
    \item  $UB - LB \le \epsilon UB$.
    \item  A feasible solution is found. 
    \item  Any subproblem~\eqref{eq:subr} is infeasible. 
\end{enumerate}
Otherwise, go to Step 4.\\
{\bf Step 4: Bounding} \\
Update the upper bounds of $\sum\limits_{k \in K} \theta_k, \sum\limits_{e \in E} \phi_e$ with~\eqref{eq:bounding}. 
\If{\eqref{eq:boundchecking} holds}{
{Replace~\eqref{Cut3} with~\eqref{eq:strongerCut3}-\eqref{eq:strongerCut3_2}}\\
}

Go to step 5.\\
{\bf Step 5: Branching} \\
\eIf{there exists a fractional component in  $(\vtheta, \vphi)$}{
Select one variable from $\vtheta, \vphi$ to according to rules in Sec.~\ref{sec:branching}\\
Create two child nodes by fixed the selected variable to $1$ and $0$ respectively. Go to step 2.\\
}
{Export the problem to a MIP solver.\\
 If the objective value is less than $UB$, update $UB$ and go to Step 3.}
\end{algorithm}
}
\section{Numerical experiments}\label{sec:numerical}
In this section, we assess the computational performance of the proposed formulation~\eqref{model:reform}, Lagrange decomposition procedure and Algorithm~\ref{alg:global}.
Results in this section illustrate the following key points.
\begin{enumerate}
\item The proposed reformulation~\eqref{model:reform} is effective.~Numerically it outperforms the McCormick formulation~\eqref{model:MC} by orders of magnitudes. 
\item The proposed Lagrange decomposition provides stronger lower bounds than the continuous relaxation bound of~\eqref{model:reform}.
\item Algorithm~\ref{alg:global} outperforms the standard B\&B algorithm of CPLEX 12.7 solver with formulation~\eqref{model:reform} by orders of magnitudes. 
\end{enumerate}

\subsection{Test instances}
To the best of our knowledge, there is no public data sets for the  virtual machine mapping problem. 
Following~\cite{mechtri2014:VM}, we randomly generate virtual request instances following rules below. 
\begin{enumerate}
\item As presented in Section~\ref{sec:background}, each virtual request
 typically involves a small number of VMs, so the size of each
 virtual request is fixed to 5. The number of virtual requests ranges from 1 to 10. In other words, we deal with problem instances up to 50 VMs.
\item For each virtual request, the communication traffic between each two VMs is uniformly generated as an integer in $\{0, 100\}$.
 For each virtual machine, the required number of CPU cores is randomly generated from 1 to 10 and the size of memory is generated from 2GB to 8GB. 
\end{enumerate}
Physical graph instances are from SND library~\cite{SNDlib}. Each node of a graph represents a server, whose number of CPU cores and memory size are
randomly chosen as an ordered pair from set $\{(8, 128), (16, 256), (32, 512),
(64, 1024)\}$. Parameters of link capacity and fixed cost are consistent with the data in SND library. 
The additional cost regarding CPU resource $A_k$ is set as $10$ for all $k \in S$.
For each O-D pair of servers, the shortest path is computed using Dijkstra's algorithm before the solution procedure. 

To measure the computational difficulty of each problem instance, we report the topology of the physical graph, the number of VMs and 
the number of binary variables. 
\subsection{Implementation and experiments setup}
All computations are implemented with C++ and all problem instances are solved by the state-of-the-art solver CPLEX 12.7 with default settings 
on a Dell Latitude E7470 laptop with Intel Core(TM) i5-6300U CPU clocked at 2.40 GHz and with 8 GB of RAM. The B\&B tree is implemented using a 
priority queue.  To have a fair and unbiased comparison, CPU time is used as computational measurement. For all the tests, the computational time is limited to 10 hours and the feasibility tolerance is set as $10^{-4}$.  The optimality tolerance for Algorithm~\ref{alg:global} is set to $0.5\%$. To improve the stability of computation, Lagrange multipliers are preserved to 4 decimal places. If no solution is available at solver termination or the solution process is killed by the solver, ($-$) is reported.
\subsection{The evaluation of formulation~\eqref{model:reform}}
We evaluate the strength of formulations~\eqref{model:RLT},~\eqref{model:reform}
with respect to the McCormick based formulation~\eqref{model:MC}. For each problem instance let $v_{MC}$ denote the optimal value of 
the continuous relaxation of~\eqref{model:MC} and $v^*$ the global optimal value. Thus the optimality gap induced by McCormick formulation 
is $\frac{v^* - v_{MC}}{v^*}$.  Formulations~\eqref{model:RLT},~\eqref{model:reform} are expected to close this optimality as more valid inequalities are added. 
We measure the closed gap via
\begin{align}\label{eq:closedgap}
 \mathrm{Closed~gap} = \frac{v-v_{MC}}{v^* - v_{MC}} \times 100
\end{align}
where $v$ represents the continuous relaxation value of~\eqref{model:RLT} or~\eqref{model:reform}. Results are summarized in Table~\ref{tab:prlt} which imply the following.
\begin{enumerate}[label=(\roman*)]
 \item The RLT inequalities close $2\%-15\%$ of the optimality gap while the combination of the RLT and~\eqref{Cut1}-\eqref{Cut3} closes $48\%-100\%$ gap.
  \item As might be expected the continuous relaxations of RLT formulation and the compact formulation~\eqref{model:reform} are computationally expensive due to the addition of valid inequalities. In particular the solution time for the continuous relaxation of~\eqref{model:reform} is becoming evidently more expensive than that of~\eqref{model:RLT} as the size of the physical network increases. 
 \end{enumerate}
\begin{table}[h!]
 \caption{Numerical evaluation for~\eqref{model:RLT},~\eqref{model:reform}}
 \label{tab:prlt}
\begin{center}
 \resizebox{\columnwidth}{!}{
\begin{tabular}{l*{9}{c}}
\toprule
&&\multicolumn{7}{c}{Continuous relaxation statistics}\\
\midrule
 ($|S|,|E|$)&\#VMs. &\multicolumn{1}{c}{\eqref{model:MC}}&&\multicolumn{2}{c}{\eqref{model:RLT}}&&\multicolumn{2}{c}{\eqref{model:reform}}\\
 \cmidrule{3-4}\cmidrule{5-6}\cmidrule{8-9}
 &&\#Cpu(Sec.)&\quad&\#CPU(Sec.)& closed gap(\%) & \quad &\#CPU(Sec.) & closed gap(\%)\\
\midrule
 (8, 10)     & 5   & 0.00  &  & 0.01  & 15.2 &  & 0.12  & \bf{84.80} \\
\smallskip
(8, 10)      & 10  & 0.02  &  & 0.01  & 11.8 &  & 0.18  & \bf{89.21} \\
\smallskip
(12, 15)     & 5   & 0.02  &  & 0.01  & 7.29 &  & 0.22  & \bf{92.28} \\
\smallskip
(12, 15)     & 10  & 0.07  &  & 0.02  & 2.37 &  & 0.43  & \bf{82.51} \\
\smallskip
(12, 15)     & 15  & 0.13  &  & 0.22  & 3.34 &  & 1.02  & \bf{95.12} \\
\smallskip
(12, 15)      & 20  & 0.27  &  & 0.72  & 3.33 &  & 1.49  & \bf{100.00} \\
\smallskip
(12, 15)      & 25  & 0.33  &  & 0.92  & 6.53 &  & 2.69 & \bf{88.11} \\
\smallskip
(12, 15)      & 30  & 0.34  &  & 1.2   & 4.89 &  & 3.18 & \bf{80.30} \\
\smallskip
(12, 15)      & 35  & 0.49  &  & 1.33   & 8.32 &  & 3.64 & \bf{94.30} \\
\smallskip
(12, 15)      & 40  & 0.67  &  & 1.42  & 7.31 &  & 3.75  & \bf{86.83} \\
\smallskip
(12, 15)      & 50  & 0.92  &  & 1.43  & 9.98 &  & 4.23  & \bf{89.83} \\
\smallskip
(15,  22)     & 10  & 0.07  &  & 0.27  & 2.33 &  & 0.22  & \bf{89.78} \\
 \smallskip
 (15,  22)     & 15  & 0.13  &  & 0.48 & 7.03 &  & 2.34  & \bf{72.56} \\
 \smallskip
(15,  22)     & 20  & 0.26  &  & 3.77  & 4.09 &  & 4.22 & \bf{86.04} \\
\smallskip
(15,  22)     & 25  & 0.85  &  & 2.75  & 4.89 &  & 5.09 & \bf{87.77} \\
\smallskip
(15, 22)     & 30  & 1.04  &  & 1.98  & 5.98&  & 4.96 & \bf{90.14} \\
 \smallskip
 (15, 22)     & 35  & 0.63  &  & 0.72  & 14.53&  & 4.57 & \bf{76.51} \\
 \smallskip
(15, 22)     & 40  & 1.25  &  &  3.44 &12.42 &  & 5.83 & \bf{79.43} \\
\smallskip
    (22, 36) & 10  & 0.92  &  & 1.46  &5.92 &  & 4.59  & \bf{100.00}\\
\smallskip
    (22, 36) & 15  & 3.01  & &3.04 & 2.51 && 16.20 & \bf{69.80}\\
\smallskip
    (22, 36) & 20  & 3.39  &  & 3.55  & 3.58 &  & 26.87 & \bf{67.75} \\
\smallskip
    (22, 36) & 25  & 6.22  &  & 4.12  & 2.90&  & 28.10 & \bf{63.64} \\
\smallskip
    (22, 36) & 30  & 8.83 &   & 3.23 &  3.91 & &29.91&  {\bf 69.12} \\
\smallskip
    (22, 36) & 35  & 8.91  &  & 3.34 &3.34 && 33.81  & \bf{53.73}\\ 
\bottomrule
\end{tabular}
}
\end{center}
\end{table}

\subsection{The Lagrange lower bounds}
Section~\ref{sec:lb}} shows that the proposed Lagrange decomposition scheme~\eqref{eq:dual} and its generalized hierarchy can close further the McCormick relaxation gap. 
We illustrate this point numerically with following settings.
 \begin{enumerate}
 \item For a given problem instance, we first partition the virtual request set $R$ to a number of subsets. Each subset has 
 at most $\delta$ requests. For instance if $\abs{R} = 7$ and $\delta = 2$, set $R$ is partition to 4 subsets 
 $\bigcurly{\{1, 2\}, \{3, 4\}, \{5, 6\}, \{7\}}$.
 \item We perform numerical experiments with $\delta \in \{1, 2, 3\}$ and solve each subproblem with CPLEX 12.7.
 \item For instances with network $(22, 36)$ it is time consuming for solving problems associated with $\delta =3$. For this reason, we skip them. 
 \end{enumerate}
The corresponding results are summarized in Table~\ref{tab:lowerbounds} and they may indicate the following.
\begin{enumerate}
 \item The Lagrange lower bound is generally stronger than the continuous relaxation bound of~\eqref{model:reform}; in particular,
it can close the optimality gap completely for certain instances in a reasonable time. For most problem instances, the proposed Lagrange lower bounding procedure closes 80\% McCormick optimality gap. 
 \item For a given instance, its Lagrange lower bound generally increases as $\delta$ increases. Correspondingly the computational time increases. 
\item There exist a couple of instances where the Lagrange lower bound with $\delta =1 $ is equal to the continuous relaxations bound (e.g.~network (15, 22) with 30VMs). For such instances, the computational cost of Lagrange lower bounding procedure is usually quit small. This is probably due to the fact that the formulation of each subproblem~\eqref{eq:subr} is strong.
\end{enumerate}
\begin{table}[h!]
 \caption{Numerical evaluation of the Lagrange lower bounds}
 \label{tab:lowerbounds}
\begin{center}
 \resizebox{\columnwidth}{!}{
\begin{tabular}{l*{9}{c}}
\toprule
($|S|,|E|$)&\#VMs. &\multicolumn{2}{c}{$\delta =1$}&&\multicolumn{2}{c}{$\delta =2$}& &\multicolumn{2}{c}{$\delta =3$} \\
 \cmidrule{3-4}\cmidrule{6-7}\cmidrule{9-10}
 \smallskip
&&\#CPU(s)&Closed gap(\%)&\quad&\#CPU(s)& Closed gap(\%)&\quad&\#CPU(s)& Closed
 gap(\%) \\
\midrule
\smallskip
 (8, 10)  & 10  & 0.91 & \bf{100.00} & & - &-  &&-&- \\
\smallskip      
 (12,15) & 10  & 8.34 & \bf{100.00} & & -  &-   &&-& -\\
\smallskip            
 (12,15) & 15  & 5.23 & \bf{100.00} & & -  &-   &&-&- \\
\smallskip
 (12,15) & 20  & 6.87 & \bf{100.00} & & -& - &&-&- \\
\smallskip            
 (12,15) & 25  & 15.78 & \bf{89.68} & & 16.31 & \bf{91.98} &&50.23& {\bf 92.12} \\
\smallskip            
 (12,15) & 30  & 12.31 & \bf{88.78} & & 25.20& \bf{88.87} &&68.38&\bf{89.34} \\
 \smallskip            
 (12,15) & 35  &12.56 &\bf{94.67}  & & 28.93 &\bf{95.03}  && 35.45&\bf{95.03} \\
\smallskip            
 (12,15) & 40  &27.56 &\bf{82.34}  & & 38.03 &\bf{83.56}  && 63.34&\bf{83.98} \\
 \smallskip            
 (12,15) & 50  &37.22 &\bf{90.12}  & & 43.92 &\bf{90.12}  && 93.40&\bf{90.24} \\
\smallskip                 
 (15, 22) & 10 &7.64 & {\bf 93.55}& &30.10& \bf{100.00} & &- &-  \\
\smallskip
(15, 22) & 15& 6.21&{\bf 76.37} & &13.34 &{\bf 76.67} & & -&-  \\
\smallskip
 (15, 22) & 20 &13.22 &{\bf 89.56} & &28.90&{\bf 90.43} &  & 90.31&\bf{92.77}  \\
\smallskip
 (15, 22) & 25 & 17.88&{\bf 91.40}& & 33.04&{\bf 91.81} &&50.23  &{\bf 96.36}  \\
\smallskip
 (15, 22) & 30 &4.11 &{\bf 90.14} & & 4.18& {\bf 90.14}& & 33.21&{\bf 93.13}  \\
 \smallskip
 (15, 22) & 35 &6.09 &{\bf 79.01} & &23.94& {\bf 80.32}& & 24.82&{\bf 80.33}  \\
 \smallskip
 (15, 22) & 40 &23.34 &{\bf 81.34} & & 24.43&{\bf 81.34} & &50.32 &{\bf 82.10}  \\
  \smallskip
 (22, 36) & 10 & 4.31&{\bf 100.00} & &- &-& &- &-  \\
\smallskip
(22, 36) & 15 & 62.31& {\bf 75.12}&  &72.12 & {\bf 81.60}&&- &-  \\
\smallskip
 (22, 36) & 20 &90.31 &{\bf 74.40} & & 139.76&{\bf 75.45}& &- &-  \\
\smallskip
 (22, 36) & 25 &95&{\bf 81.36} & & 139.12&{\bf 83.74}& & -&  -\\
\smallskip
 (22, 36) & 30 & 150&{\bf 70.12} & & 223&{\bf 73.20} & & -&-  \\ 
 \smallskip
  (22, 36) & 35 & 158.18&{\bf 60.56} && 413.93& {\bf 64.32}& & -& -  \\ 
\bottomrule
\end{tabular}
}
\end{center}
\end{table}

\subsection{The evaluation of Algorithm~\ref{alg:global}}
In this section, we evaluate the performance of Algorithm~\ref{alg:global} in
comparison with model~\eqref{model:MC} and model~\eqref{model:reform} in terms of solution time and B\&B nodes. For this purpose we setup numerical experiments as follows.
\begin{enumerate}
\item For each problem instance we try to solve it to global optimality within a time limit of 10 hours with formulation~\eqref{model:MC} and~\eqref{model:reform}. 
Both CPU time and the number of B\&B nodes of CPLEX 12.7 are recorded. We also implement branching priority strategy over $\vtheta, \vphi$ using a 
\texttt{BranchCallback} for~\eqref{model:reform}.
\item For each problem instance we solve it to global optimality with Algorithm~\ref{alg:global}. 
 Lagrange cuts are generated using the single request based decomposition~\eqref{eq:dual}.
The upper bound tolerance parameter $\delta$ is set as $5\%$ and the optimality tolerance is $0.5\%$. 
\item The heuristic algorithm~\ref{alg:heuristic} is used in Algorithm~\ref{alg:global} for the generation of upper bounds. Its input parameter $n$ is set to  $\floor{R/2}$.
\item As highlighted before our B\&B algorithm uses a heuristic for generating upper bounds whose quality might influence the overall computational time. 
In order to facilitate a fair comparison we solve each problem instance $10$ times and take the respective averages of CPU time and B\&B nodes as measurement.  
\end{enumerate}

Numerical results are summarized in Table~\ref{tab:decomposition} and we make some comments below.
\begin{enumerate}
\item For problem instances with 10 VMs, all three approaches can solve the problem to global optimality.
The compact formulation~\eqref{model:reform} performs the best in terms of computational efficiency.  

\item For problem instances with more than $30$ requests, the McCormick formulation~\eqref{model:MC} is the most
time-consuming approach while Algorithm~\ref{alg:global} is computationally most efficient. For some small instances (e.g. (12, 15) 20 VMs), formulation~\eqref{model:reform} 
performs better than Algorithm~\ref{alg:global}.~This is probably due to the fact that CPLEX sometimes finds better upper bound than Algorithm~\ref{alg:heuristic}.
 \item For problem instances having more than $40$ VMs, CPLEX cannot solve the
problem to optimality within the $10$-hour time limit even using the compact model~\eqref{model:reform}.
In contrast Algorithm~\ref{alg:global} provides optimal solutions much faster.  
For most problem instances, Algorithm~\ref{alg:global} is at least 10 times and sometimes 30 times more efficient than the CPLEX default 
branch and bound algorithm with Model~\eqref{model:reform}. 

\item Given a set of VMs, formulation~\eqref{model:reform} and Algorithm~\ref{alg:global} become less advantageous in terms of computational efficiency as the size of the physical network increases. This is largely due to two reasons. First our subproblem~\eqref{eq:subr} is not separable in terms of physical network components thus making the lower bounding procedure computationally more expensive. Second Algorithm~\ref{alg:heuristic} becomes less effective in finding good upper bounds leading to more branch nodes. 
\item The current implementation of Algorithm~\ref{alg:global} fails to solve problems instances with more than $700$ variables due to memory issues. 
Numerically we observed that the implementation incurs memory issues when the number of branch nodes is over 1000. 
\end{enumerate}
\begin{table}[!h]
 \caption{The evaluation of Algorithm~\ref{alg:global}}
 \label{tab:decomposition}
\begin{center}
\begin{tabular}{lcc*{9}{r}}
\toprule
 ($|S|,|E|$)&\#VMs. &\#Bin. &\multicolumn{2}{c}{\eqref{model:MC}} &&\multicolumn{2}{c}{\eqref{model:reform}} & ~&\multicolumn{2}{c}{ Alg.~\ref{alg:global}} \\
  \cmidrule{4-5}\cmidrule{7-8}\cmidrule{10-11} 
  &&& CPU & Nodes & & CPU &Nodes & & CPU &Nodes \\
\midrule
\smallskip        
(8, 10)    & 5  & 68  & 0.86           & 0     &  & {\bf 0.29}  & 0   &  & 0.58 & 0\\
\smallskip
(8, 10)    & 10 & 108 & 29.1           & 229   &  & 1.46        & 3   &  & {\bf 0.91}& 0 \\
\smallskip
(12, 15)   & 5  & 102 & 1.63           & 16    &  & {\bf 0.95}  & 0  &  & { 1.21} & 0\\
\smallskip
(12, 15)   & 10 & 162 & 123.31         & 8897  &  & {\bf 4.31}     & 0  &  &  8.87 & 12 \\
\smallskip
(12, 15)   & 15 & 222 & 321.32         & 9873  &  & 56.27       & 56  &  & {\bf 22.53} &22 \\
\smallskip
(12, 15)   & 20 & 282 & 1035.32        & 10766 &  &  {\bf 1.85}     & 0 &  &  3.21 &2\\
\smallskip
(12, 15)   & 25 & 342 & 32492.00  & 10645 &  & 270.32     & 240 &  & {\bf 20.54}& 6 \\
\smallskip
(12, 15)   & 30 & 402 & -              & -     &  & 358.89   &197 & & {\bf 59.50}&32 \\
\smallskip
(12, 15)   & 35 & 447 & -              &    -   &  & 7684.81        & 6534    &  & {\bf 169.00} & 29\\
\smallskip
(12, 15)   & 40 & 507 & -              &     -  &  & -           & -   &  & {\bf 534.22} & 138\\
\smallskip
(12, 15)   & 50 & 627 & -              &      - &  & -           &  -   &  & {\bf 1340.18} & 82\\
\smallskip
(15, 22)   & 10 &  187   &         3781         &    2383  && {\bf 21.39}&    20         &     & 24.01 & 3\\
\smallskip
(15, 22)   & 15 & 262    &         -        &   -    &&  2366.72  &    6414       &     &{\bf 123.01}  & 12 \\
\smallskip
(15, 22)   & 20 & 337  &       -         &   -    &  &   1702.9 &2264              &  & {\bf 239.22} & 49\\
\smallskip
(15, 22)   & 25 & 412 &          -      &   -    &     &11400.8&12011             && {\bf 304.44} & 46\\
\smallskip
(15, 22)   & 30 &  487   &         -       & -      & &  12103.3        &2750&    &{\bf 406.92} &24 \\
\smallskip
(15, 22)   & 35 & 562    &        -        &     -  &  &     -        & -    && {\bf 731.12} &134 \\
\smallskip
(15, 22)   & 40 & 637    & -              &   -    &  & -           & -    &  &{\bf 2649.20}& 316\\
\smallskip
(15, 22)   & 50 &  787   & -              &     -  &  & -           &  -   &&- & -\\
\smallskip
(22, 36)   & 10 &  228   & 10982            & 9948    &  & {\bf 4.43}        & 0    && { 5.21} &1 \\
\smallskip
(22, 36)   & 15&  388   & -              &      - &  &     {208.23}      & 42&& {\bf 98.21} &19 \\
\smallskip
(22, 36)   & 20 &   498  & -              &     -  &  & -           &-     &&{\bf 763.33}  &394 \\
\smallskip
(22, 36)   & 25 &   608  & -              &      - &  & -           &    - &&{\bf 3381.22}  & 498\\
\smallskip
(22, 36)   & 30 &  718   & -              &      - &  & -           &  -  && {\bf 8382.12} &913 \\
\smallskip
(22, 36)   & 40 &     938& -              &     -  &  & -           &   -  & &- &-\\
\midrule
\end{tabular}
\end{center}
\end{table}

%
\section{Conclusion}\label{sec:conclusion}
This work proposes a couple of mathematical programming based algorithms for the
optimal mapping of virtual machines while taking into account
the bilinear bandwidth constraints and other knapsack constraints regarding CPU
and memory.~The first one is a compact model involving RLT inequalities and some
strong valid inequalities exploiting the problem structure. The second one is a
Lagrange decomposition based B\&B algorithm for solving larger problem instances. We
show both theoretically and numerically that the proposed valid inequalities and bounding procedures can
improve the continuous relaxation bounds significantly. 
We also demonstrate that the proposed B\&B algorithm is numerically encouraging. 

Based on the results presented in this paper, several research directions can be considered. 
First, decomposition strategies exploiting the structure of the physical network should be investigated 
to accelerate the solution procedure of problem instances associated with a large physical network. 
Second, specialized branch-and-cut algorithm incorporating the Lagrange lower bounding procedure 
and some recent findings of the convex and concave estimators in~\cite{Wang2017:convex} can be devised. 
Third, some approximation algorithm might be devised to achieve guaranteed feasible solutions of high quality.
\bibliographystyle{spmpsci} 
\bibliography{phd}
\end{document}